\documentclass[10pt,reqno]{amsart}

\usepackage[a4paper,left=25mm,right=25mm,top=25mm,bottom=25mm,marginpar=25mm]{geometry}

\usepackage{amsfonts}
\usepackage{amsthm}
\usepackage{amssymb}
\usepackage{amsmath}
\usepackage{amscd}
\usepackage{mathtools}
\usepackage{color}
\usepackage{enumerate}
\usepackage{dsfont}
\usepackage{bm}
\usepackage[latin1]{inputenc}
\usepackage[displaymath,mathlines]{lineno}
\usepackage[mathscr]{eucal}
\numberwithin{equation}{section}
\usepackage{epstopdf} 
\usepackage[unicode=true]{hyperref}
\usepackage{indentfirst}
\usepackage{lineno}
\usepackage{centernot}
\setpagewiselinenumbers
 
 \makeindex

\theoremstyle{plain}
\newtheorem{theorem}{Theorem}[section]
\newtheorem{lemma}[theorem]{Lemma}
\newtheorem{corollary}[theorem]{Corollary}
\newtheorem{proposition}[theorem]{Proposition}

\theoremstyle{definition}
\newtheorem{definition}[theorem]{Definition}
\newtheorem{remark}[theorem]{Remark}
\newtheorem{example}[theorem]{Example}

\newcommand{\ap}{\alpha_p}
\newcommand{\apto}{\xrightarrow{\ap}}

 \email{nuno.januario.alves@univie.ac.at}
 \email{paulos.pjoo@gmail.com}

\begin{document}
  \numberwithin{equation}{section}
\title[]{A mode of convergence arising in diffusive relaxation}

\author[Nuno J. Alves and Jo\~{a}o Paulos]{}
\address[Nuno J. Alves]{
      Faculty of Mathematics, University of Vienna, Oskar-Morgenstern-Platz 1, 1090 Vienna, Austria.}

\keywords{modes of convergence, measurable functions, diffusive relaxation}
\subjclass[2020]{28A20, 35Q35} 


\maketitle

\centerline{\scshape Nuno J. Alves \& Jo\~{a}o Paulos}

\begin{abstract}
In this work, a mode of convergence for measurable functions is introduced. A related notion of Cauchy sequence is given and it is proved that this notion of convergence is complete in the sense that Cauchy sequences converge. Moreover, the preservation of convergence under composition is investigated. The origin of this mode of convergence lies in the path of proving that the density of a Euler system converges almost everywhere (up to subsequences) towards the density of a non-linear diffusion system, as a consequence of the convergence in the relaxation limit. 
 \end{abstract}


\section{Introduction}
In the context of diffusive relaxation, it has been proved that a family $\big((\rho_\varepsilon, u_\varepsilon)\big)_{\varepsilon > 0},$ where $(\rho_\varepsilon, u_\varepsilon)$ is a dissipative weak solution of the Euler system with friction 
\begin{equation*} 
\begin{cases}
 \partial_t\rho_\varepsilon + \nabla \cdot (\rho_\varepsilon u_\varepsilon) = 0, \\
 \partial_t(\rho_\varepsilon u_\varepsilon)+ \nabla \cdot (\rho_\varepsilon u_\varepsilon \otimes u_\varepsilon) = -\dfrac{1}{\varepsilon} \rho_\varepsilon \nabla \dfrac{\delta \mathcal{E}}{\delta \rho}(\rho_\varepsilon) -\dfrac{1}{\varepsilon} \rho_\varepsilon u_\varepsilon  , \\
 \end{cases}
\end{equation*}
converges in the sense of relative energy to a strong and bounded away from vacuum solution $\bar \rho$ of the non-linear diffusion system
\begin{equation*} 
\partial_t \bar \rho=\nabla \cdot \Big(\bar \rho \nabla\dfrac{\delta \mathcal{E}}{\delta \rho}(\bar \rho)\Big) .\\
\end{equation*}
The systems of equations above are defined in a space-time domain $(0,T) \times \Omega,$ where $T>0$ is a fixed time horizon and $\Omega$ is either the $d$-dimensional torus $\mathbb{T}^d$ or a smooth bounded domain of $\mathbb{R}^d,$ $d \in \mathbb{N}.$ Moreover, the functional $\mathcal{E}$ is defined as
\begin{equation*}
\mathcal{E}(\rho) := \int_\Omega h(\rho) \, dx  ,
\end{equation*}
where $h$ is the internal energy function. \par 
The convergence is established by means of the relative energy $\Psi_\varepsilon : \mathopen{[}0,T) \to \mathbb{R},$ given by 
 \[\Psi_\varepsilon(t) = \int_\Omega \varepsilon \tfrac{1}{2} \rho_\varepsilon |u_\varepsilon - \bar u|^2 + h(\rho_\varepsilon | \bar \rho) \, dx , \quad \bar u = - \nabla\dfrac{\delta \mathcal{E}}{\delta \rho}(\bar \rho) , \] where \[h(\rho | \bar \rho) = h(\rho) - h(\bar \rho) - h^\prime(\bar \rho)(\rho-\bar \rho)  . \]
Specifically, assuming that $\Psi_\varepsilon(0) \to 0$ as $\varepsilon \to 0,$ then
\begin{equation*} 
\sup\limits_{[0,T[} \Psi_{\varepsilon} \to 0 \ \text{as} \ \varepsilon \to 0 .
\end{equation*}
In particular, for each $t \in [0,T),$ it holds that 
\begin{equation} \label{hconv}
\int_\Omega h(\rho_\varepsilon | \bar \rho) \, dx \to 0 \ \text{as} \ \varepsilon \to 0 . 
\end{equation} 
For more details refer to \cite{alves, carrillo, choi, gasdynamics, thanos}.  \par 
At this point, it is relevant to examine if (\ref{hconv}) implies that the sequence $(\rho_\varepsilon)$ converges almost everywhere to $\bar \rho,$ even if one has to pass to a subsequence. In the quest to reach that conclusion, one stumbles upon a new notion of convergence. The intent of this work is to explore that mode of convergence at the generality of measurable functions on a general measure space. The appearance of that notion of convergence is described next. \par
Within the framework of the analysis described above, the function $h$ is assumed to belong to the class $\mathcal{H}_k^\gamma$ given by  
\[\mathcal{H}_k^\gamma = \Big\{h \in C^2((0,\infty)) \cap C([0,\infty)) \ \Big| \ \lim\limits_{\rho \to \infty} \frac{h(\rho)}{\rho^\gamma}=\frac{k}{\gamma-1} , \ h^{\prime \prime}(\rho) > 0 \ \forall \rho>0 \Big\} ,\] 
where $k > 0$ and $\gamma > 1.$ A typical example of a function belonging to $\mathcal{H}_k^\gamma$ is $h(\rho) = \frac{k}{\gamma-1}\rho^\gamma.$  Regarding that class of functions, recall the following lemma \cite{thanos}:
\begin{lemma} \label{hlemma} 
 Let $h \in \mathcal{H}_k^\gamma$ with $k >0$ and $\gamma > 1,$ and let $\bar \rho \in [\delta,M]$ for some $\delta > 0$ and $M < \infty.$ Then, there exists $R \geq M{+}1$ and positive constants $C_1, C_2$ depending on $k,\gamma, \delta, M$ such that
\begin{equation*}
 h(\rho| \bar \rho) \geq
 \begin{cases}
  C_1 |\rho-\bar \rho|^2  , \quad & \text{if} \ \rho \in  [0,R]  , \\

  C_2 |\rho-\bar \rho|^{\gamma} , \quad & \text{if} \ \rho \in \ ]R,\infty [ .

 \end{cases}
\end{equation*}
\end{lemma}
The next step is to combine (\ref{hconv}) with the previous lemma. Let $\varepsilon = 1/n, \ n \in \mathbb{N}, \ \rho_n = \rho_\varepsilon$ and define $B_n=B_n(t)$ by 
 \[B_n = \{x \in \Omega \ | \ 0 \leq \rho_n(t,x) \leq R \} , \] where $R$ is as in the lemma. Thus,
\begin{equation*}
\begin{split}
\int_{\Omega} h(\rho_n | \bar \rho)  \, dx & = \int_{B_n} h(\rho_n | \bar \rho)  \, dx + \int_{B_n^c} h(\rho_n | \bar \rho)  \, dx \\
& \geq C_1 \int_{B_n} |\rho_n - \bar \rho|^{2} \, dx + C_2 \int_{B_n^c} |\rho_n - \bar \rho|^{\gamma } \, dx \\
& > C_1 \int_{B_n} |\rho_n - \bar \rho|^{2} \, dx + C_2|R-M|^{\gamma } \mathcal{L}(B_n^c)\\
& \geq C_1 \int_{B_n} |\rho_n - \bar \rho|^{2} \, dx + C_2\mathcal{L}(B_n^c) ,
\end{split}
\end{equation*}
where $B_n^c$ denotes the complement of $B_n$ in $\Omega$ and $\mathcal{L}$ is the Lebesgue measure. Hence, from $(\ref{hconv})$ one can conclude that
\begin{equation} \label{motivationalpha}
\int_{B_n} |\rho_n - \bar \rho|^{2} \, dx \to 0  , \quad  \mathcal{L}(B_n^c) \to 0  , \quad \text{as} \ n \to \infty .
\end{equation}
This motivates a notion of convergence in which a sequence of measurable functions $(\rho_n)$ converges to a measurable function $\bar \rho$ if there exists a sequence of measurable sets $(B_n)$ whose measure of the complements tends to zero as $n \to \infty$ and such that the integral of $|\rho_n-\bar \rho|^p$ over $B_n$ goes to zero as $n \to \infty,$ where $p \geq 1$; see Definition \ref{alphaconv}. As one will see, this notion of convergence implies convergence in measure, which is a sufficient condition for the existence of a subsequence that converges almost everywhere \cite{bartleelements}. \par
Section \ref{modessection} presents the notions of convergence that are treated in this work, together with the relations between them. Two notions of convergence are introduced. The first, the one derived in the analysis above, represents the focus of this work. The second, which is considered here for its similarity to the first, turned out to be very useful to prove a completeness result.  \par
Section \ref{cauchysection} introduces notions of Cauchy sequences associated with the convergences studied in this work. Moreover, completeness theorems related to those Cauchy notions are obtained. \par
Section \ref{convergencesection} is devoted to understand how different modes of convergence become equivalent under extra assumptions. As a consequence of this study, it is proved that the space of measurable functions, together with the notion of convergence motivated by (\ref{motivationalpha}), form a sequential convergence class \cite{dudley, kelley}. \par
Section \ref{preservationsection} is inspired by \cite{bartlepreserv} and explores in which conditions a function preserves (under composition) the modes of convergence treated in this work. For these modes of convergence, one finds that such preservation holds for Lipschitz continuous functions. Furthermore, it turns out that this characterizes Lipschitz continuity. Moreover, it is proved that convergence in $L_p$ is also preserved by Lipschitz continuous functions, at the generality of measurable functions, which is different from a result obtained in \cite{bartlepreserv}, where the same investigation is done considering sequences in $L_p.$ 

\section*{Notation and conventions}
Unless stated otherwise, all considered functions belong to a general measure space $(X, \mathbb{X} , \mu).$ The elements of $\mathbb{X}$ are called measurable sets, and by a measurable function one understands a $\mathbb{X}$-measurable real-valued function defined on $X$. The space of measurable functions is denoted by $M(X).$ Given a sequence $(f_n),$ one writes $(g_n) \subseteq (f_n)$ to indicate that $(g_n)$ is a subsequence of $(f_n).$ In the examples, one considers $(X, \mathbb{X} , \mu)=(A, \mathcal{B}, \mathcal{L}),$ where $A$ is a subset of $\mathbb{R},$ $\mathcal{B}$ is the Borel $\sigma$-algebra and $\mathcal{L}$ is the Lebesgue measure. The composition of a function $\varphi: \mathbb{R} \to \mathbb{R}$ and a function $f:X \to \mathbb{R}$ is denoted by $\varphi(f).$ The set of positive integers is denoted by $\mathbb{N}$ while the set of non-negative integers is denoted by $\mathbb{N}_0.$ Moreover, the integrability exponent $p$ is a fixed real number belonging to $[1, \infty) .$

\section{Modes of convergence} \label{modessection}
Within the realm of measurable functions, many are the ways of saying that a sequence converges to a certain limit. Among the most important convergence notions, one encounters the classical notions of convergence in $L_p,$ in measure, almost uniformly (a.u.) or almost everywhere (a.e.). In this manuscript, two other notions of convergence for measurable functions are introduced. The first, named $\alpha_p$-convergence, is the main focus of this work, and has its definition motivated by the analysis detailed above. The second, convergence almost in $L_p,$ arises as a combination of the $L_p$ and almost uniformly convergences, and it is considered here due to its similarity to the $\alpha_p$-convergence. \par 
It is always the case that the following relations hold
\[L_p\text{-convergence} \Rightarrow \text{convergence almost in} \ L_p  \Rightarrow \alpha_p\text{-convergence} \Rightarrow  \text{convergence in measure}  .\] 
Interestingly enough, none of this modes of convergence are equivalent. Moreover, in a finite measure space one has that
\[\text{convergence a.e.} \Rightarrow \text{convergence a.u.}  \Rightarrow \text{convergence almost in} \ L_p \Rightarrow \alpha_p\text{-convergence}.\] \par 

\subsection{Preliminaries}
First, one recalls some basic definitions and results from the classical theory. It is customary to  define convergence in $L_p$ for sequences of functions in $L_p,$ however here one considers measurable functions in general. \par 
A sequence of measurable functions $(f_n)$ is said to converge to a measurable function $f$:
\begin{enumerate}[(i)]
\item in $L_p$ if \[\int_X |f_n - f|^p \, d\mu \to 0  \ \text{as} \ n \to \infty  ,\] 
 \item in measure if 
 \[ \forall \delta > 0 \ \mu\big(  \{x \in X \ | \ |f_n(x) - f(x)| \geq \delta \}  \big) \to 0  \ \text{as} \ n \to \infty  ,\] 
 \item almost uniformly if 
 \[\forall \delta >0 \ \exists E_\delta \in \mathbb{X} \ \text{with} \ \mu(E_\delta) < \delta \ \text{such that} \ f_n \to f \ \text{uniformly} \ \text{on} \ E_\delta^c  ,\]
 
 \item almost everywhere if 
 \[\exists N \in \mathbb{X} \ \text{with} \ \mu(N) = 0 \ \text{such that} \ f_n \to f \ \text{pointwise on} \ N^c  . \]
\end{enumerate}
A common feature of these modes of convergence is the uniqueness almost everywhere of the limit, that is, if a sequence $(f_n)$ converges to $f$ and $g$  in one the previous modes, then $f = g$ almost everywhere. \par In general, some notions of convergence imply others. For instance, if $(f_n)$ converges to $f$ in $L_p$ then it converges to $f$ in measure since 
\[\delta^p \mu\big(  \{x \in X \ | \ |f_n(x) - f(x)| \geq \delta \}  \big)  \leq \int_X |f_n - f|^p \, d\mu  .\]
Furthermore, if $(f_n)$ converges almost uniformly to $f$ then $(f_n)$ converges to $f$ in measure. On the other hand, if $(f_n)$ converges to $f$ in measure then there exists a subsequence $(f_{n_k})$ that converges to $f$ almost uniformly (and hence almost everywhere) \cite{bartleelements}.

\subsection{\texorpdfstring{$\bm{\ap}$}{\unichar{"3B1}\unichar{"209A}}-Convergence}
\begin{definition} \label{alphaconv}
A sequence $(f_n)$ of measurable functions is said to $\ap$-converge to a measurable function $f$ if there exists a sequence of measurable sets $(B_n)$ with $\mu (B_n^c) \to 0$ as $n \to \infty$ such that 
\[\int_{B_n} |f_n - f|^p \, d\mu \to 0 \ \text{as} \ n \to \infty .\] In this case one writes $f_n \apto f .$
\end{definition}
It is clear that convergence in $L_p$ implies $\ap$-convergence (simply take $B_n = X$ for every $n \in \mathbb{N}$), however, the reciprocal is not true in general. 
\begin{example} \label{examplealphanotLp}
Let $(X, \mathbb{X} , \mu) = ([0,1], \mathcal{B}, \mathcal{L})$ and $f_n = n^{1/p} \chi_{[0,1/n]}.$ Then $(f_n)$ $\ap$-converges to $0$ but it does not converge to $0$ in $L_p$. Indeed, setting $B_n = ]1/n,1]$ one has that \[\int_{B_n} |f_n|^p \, d\mu = \int_{[0,1/n]\cap]1/n,1]} n \ d\mu = 0 , \]
but \[\int_X |f_n|^p \, d\mu = 1  . \]
\end{example}
As a first result, one shows that the $\ap$-convergence is linear. 
\begin{proposition}
If $(f_n),(g_n)$ $\ap$-converge to $f,g,$ respectively, and $a,b \in \mathbb{R}$ then 
$(af_n + bg_n)$ $\ap$-converges to $af+bg. $
\end{proposition}
\begin{proof}
Let $(B_n),(D_n)$ be the sequences of measurable sets associated with the $\ap$-convergence of $(f_n),(g_n),$ respectively. Let $C_n = B_n \cap D_n$ so that $\mu(C_n^c) \leq \mu(B_n^c) + \mu(D_n^c) \to 0$ as $n \to \infty.$ The result follows from the inequality 
\[\int_{C_n} |af_n+bg_n-af-bg|^p \, d\mu \leq 2^{p-1}|a|^p \int_{B_n} |f_n-f|^p \, d\mu + 2^{p-1}|b|^p \int_{D_n} |g_n-g|^p \, d\mu  . \]
\end{proof}
Next, it is shown that the $\ap$-convergence implies convergence in measure. 
\begin{proposition} \label{alphaimplymeasure}
Let $(f_n)$ be a sequence of measurable functions that $\ap$-converges to a measurable function $f$. Then $(f_n)$ converges to $f$ in measure. 
\end{proposition}
\begin{proof}
Suppose, towards a contradiction, that 
\[\exists\delta > 0  \quad \exists \varepsilon > 0 \quad \forall n\in \mathbb{N} \quad \exists k_n \geq n \quad \mu(E_{k_n}^\delta) \geq \varepsilon  ,\]
where \[E_{k_n}^\delta = \{x \in X \ | \ |f_{k_n}(x) -f(x)| \geq \delta \}  .\] By hypothesis, there exist a sequence of measurable sets $(B_n)$ and $N_1,N_2 \in \mathbb{N}$ such that 
\[\forall k \geq N_1 \quad \int_{B_k} |f_k-f|^p \, d\mu < \frac{\delta^p \varepsilon}{2} , \qquad \forall k \geq N_2 \quad \mu(B_k^c) < \frac{\varepsilon}{2} .\]
Let $N = \max\{N_1,N_2\}.$ There exists $k_N \geq N$ such that $\mu(E_{k_N}^\delta) \geq \varepsilon.$ Then
\[\varepsilon \leq  \mu(E_{k_N}^\delta) = \mu\big(E_{k_N}^\delta \cap B_{k_N})+ \mu(E_{k_N}^\delta \cap B_{k_N}^c) < \mu(E_{k_N}^\delta \cap B_{k_N}) + \frac{\varepsilon}{2}  ,\]
whence \[\mu(E_{k_N}^\delta \cap B_{k_N}) > \frac{\varepsilon}{2}  .\]
Consequently,
\[\int_{B_{k_N}} |f_{k_N}-f|^p \, d\mu \geq \ \int_{E_{k_N}^\delta \cap B_{k_N}} |f_{k_N}-f|^p \, d\mu \geq \delta^p \mu(E_{k_N}^\delta \cap B_{k_N}) >  \frac{\delta^p \varepsilon}{2}  ,\]
a contradiction.
\end{proof}
As a simple corollary of the previous result one has the following:
\begin{corollary}
If $f_n \apto f$ and $f_n \apto g,$ then $f = g$ almost everywhere.
\end{corollary}
Even though the $\ap$-convergence implies convergence in measure, the reciprocal is not, in general, true. 
\begin{example} \label{examplemeasurenotalpha}
Let $(X, \mathbb{X} , \mu) = ([0,\infty), \mathcal{B}, \mathcal{L})$ and $f_n = \frac{1}{n^{1/p}} \chi_{[0,n]}.$ One can readily see that $(f_n)$ converges to $0$ in measure. However, it does not $\ap$-converge to $0$. In order to see that, let $(B_n)$ be a sequence of measurable sets such that $\mu(B_n^c) \to 0$ as $n \to \infty,$ and notice that 
\begin{equation*}
\begin{split}
\int_{B_n} f_n^p \, d\mu & = \frac{1}{n} \mu ([0,n] \cap B_n) 
\\ & = \frac{1}{n} \big(\mu ([0,n]) - \mu ([0,n] \cap B_n^c) \big) \\
& = 1 - \frac{1}{n}\mu ([0,n] \cap B_n^c) \\
& \to 1 \  \text{as} \ n \to \infty  ,
\end{split}
\end{equation*}
since \[\frac{1}{n}\mu ([0,n] \cap B_n^c) \leq \frac{1}{n}\mu (B_n^c) \to 0 \ \text{as} \ n \to \infty  .\]
\end{example}

\begin{remark}
The sequences in Examples \ref{examplealphanotLp} and \ref{examplemeasurenotalpha} are some of the canonical examples of sequences that converge in measure but not in $L_p$ to a certain limit. In light of the $\ap$-convergence, one has a way of distinguishing these two examples as one $\ap$-converges while the other does not.
\end{remark}

\subsection{Convergence almost in \texorpdfstring{$\bm{L_p}$}{L\unichar{"209A}}}

\begin{definition}
A sequence $(f_n)$ of measurable functions is said to converge almost in $L_p$ to a measurable function $f$ if for each $\delta > 0$ there exists a measurable set $E_\delta$ with $\mu(E_\delta) < \delta$ such that 
\[ \int_{E_\delta^c} |f_n - f|^p \, d\mu \to 0 \ \text{as} \ n \to \infty  . \]

\end{definition}
It is clear that convergence in $L_p$ implies convergence almost in $L_p$ (take $E_\delta = \emptyset$ for any $\delta > 0$), but the reciprocal is not true in general (consider the sequence in Example \ref{examplealphanotLp}). Furthermore, arguing by contraction as in the proof of Proposition \ref{alphaimplymeasure}, one can prove that convergence almost in $L_p$ implies convergence in measure. For an example in which a sequence converges in measure but not almost in $L_p$ consider Example \ref{examplemeasurenotalpha}. Convergence almost in $L_p$ is also the natural mode of convergence for almost in $L_p$ spaces introduced in \cite{bravo2012optimal,calabuig2019representation} .\par
Next, one wishes to explore the relationship between the $\ap$-convergence and the convergence almost in $L_p.$ First, it is proved that almost in $L_p$ convergence implies $\ap$-convergence.

\begin{proposition} \label{almostLimpliesalpha}
Let $(f_n)$ be a sequence of measurable functions that converges almost in $L_p$ to a measurable function $f$. Then $(f_n)$ $\ap$-converges to $f$. 
\end{proposition}
\begin{proof}
Since $(f_n)$ converges to $f$ almost in $L_p,$ for each $k \in \mathbb{N}$ there exist a measurable set $E_k$ with $\mu(E_k) < 1/k$ and a number $N(k) \in \mathbb{N}$ such that for each $n \geq N(k)$ one has \[\int_{E_k^c} |f_n - f|^p \, d\mu < \frac{1}{k}  . \]
One can assume, without loss of generality, that $N(k+1) > N(k) \ \forall k \in \mathbb{N}.$ For each $n \in \mathbb{N}$ let $B_n = E_1^c$ if $n \leq N(2)$ and $B_n = E_k^c$ if $N(k) < n \leq N(k+1) , \ k \in \mathbb{N} \setminus \{ 1\} . $ It is clear that $\mu(B_n^c) \to 0$ as $n \to \infty.$ \par 
Let $\varepsilon > 0$ and let $k \in \mathbb{N} \setminus \{1 \}$ be such that $1/k < \varepsilon.$ For $n > N(k)$ one has that 
\[\int_{B_n} |f_n - f|^p \, d \mu = \int_{E_{k+j}^c} |f_n - f|^p \, d\mu,  \]
where $j \in \{0,1,2, \ldots \}$ is such that $N(k+j) < n \leq N(k+j+1).$ Consequently, for $n > N(k),$
\[\int_{B_n} |f_n - f|^p \, d \mu < \frac{1}{k+j} \leq \frac{1}{k} < \varepsilon , \]
which concludes the proof.
\end{proof}
The reciprocal of the previous proposition is not true in general. 
\begin{example}
Let $(X, \mathbb{X} , \mu) = ([0,1], \mathcal{B}, \mathcal{L})$ and for each $n \in \mathbb{N}$ let $k(n), j(n) \in \mathbb{N}_0$ be such that $n = 2^{k(n)}+j(n), \ 0 \leq j(n) < 2^{k(n)},$ and set $F_n = [j(n)/2^{k(n)},(j(n)+1)/2^{k(n)}].$ Let $f_n = 2^{k(n)/p} \chi_{F_n}.$ Then $(f_n)$ $\ap$-converges to $0$ (take $B_n = F_n^c$), but it does not converge to $0$ almost in $L_p$. In order to verify the latter, let $\delta = 1/10^3$ and let $E_\delta$ be any measurable set with $\mu(E_\delta) < 1/10^3.$ For every $N \in \mathbb{N}$ one can always find $n \geq N$ such that 
\[\int_{E_\delta^c} f_n^p \, d\mu = 2^{k(n)} \mu(F_n \cap E_\delta^c) = 1  ,\]
therefore convergence almost in $L_p$ does not hold.
\end{example}
Even though $\ap$-convergence does not imply convergence almost in $L_p,$ the result holds if one passes to a subsequence. 
\begin{proposition} \label{alphaimpliesalmostsubsequence}
Let $(f_n)$ be a sequence of measurable functions that $\ap$-converges to a measurable function $f$. Then there exists a subsequence $(f_{k_n})$ that converges almost in $L_p$ to $f$. 
\end{proposition}
\begin{proof}
Let $(B_n)$ be the sequence of measurable sets associated with the $\ap$-convergence of $(f_n)$. For each $n \in \mathbb{N}$ there exists $k_n \geq n$ such that $\mu(B_{k_n}^c) < 1/2^n.$ Set $C_n = \bigcap_{i=n}^\infty B_{k_i}$ and notice that $C_n \subseteq C_{n+1}$ $\forall n \in \mathbb{N}.$ Moreover, 
\[\mu(C_n^c) \leq \frac{1}{2^{n-1}} \to 0 \ \text{as} \  n \to \infty  ,\] and \[\int_{C_n} |f_{k_n}-f|^p \, d\mu \leq  \int_{B_{k_n}} |f_{k_n}-f|^p \ d\mu \to 0 \ \text{as} \ n\to \infty  .\]
Next it is shown that $(f_{k_n})$ converges to $f$ almost in $L_p$. Let $\delta > 0$ and since $\mu(C_n^c) \to 0$ as $n \to \infty$ there exists $N \in \mathbb{N}$ such that $\mu(C_N^c) < \delta.$ Set $E_\delta = C_N^c$ and let $\varepsilon > 0.$ There exists $K \in \mathbb{N}$ such that 
\[\forall n \geq K \quad \int_{C_n} |f_{k_n}-f|^p \ d\mu < \varepsilon  .\]
Thus, for $n \geq \max\{N,K\},$ $C_N \subseteq C_n$ and 
\[\int_{E_\delta^c} |f_{k_n}-f|^p \ d\mu = \int_{C_N} |f_{k_n}-f|^p \ d\mu \leq \int_{C_n} |f_{k_n}-f|^p \ d\mu < \varepsilon ,\]
therefore $(f_{k_n})$ converges to $f$ almost in $L_p$.
\end{proof}

\subsection{Case \texorpdfstring{$\bm{\mu(X) < \infty}$}{\unichar{"3BC}(X) < \unichar{"221E}}}
$ $\newline
A classical theorem of Egorov says that in a finite measure space, convergence almost everywhere implies convergence almost uniformly. Moreover, it is clear that in a finite measure space, convergence almost uniformly implies convergence almost in $L_p.$ Combining this with Proposition \ref{almostLimpliesalpha} yields:

\begin{proposition}
Let $(f_n),f$ be measurable functions on a finite measure space. If $(f_n)$ converges to $f$ almost everywhere, then $(f_n)$ $\ap$-converges to $f$.
\end{proposition}

\section{Cauchy sequences and completeness theorems} \label{cauchysection}
\begin{definition}
A sequence $(f_n)$ of measurable functions is said to be:
\begin{enumerate}[(i)]
\item $\alpha_p$-Cauchy if there exists a sequence of measurable sets $(B_n)$ with $\mu (B_n^c) \to 0  \ \text{as} \ n\to \infty $ such that 
\[\int_{B_n \cap B_m} |f_n - f_m|^p \, d\mu \to 0 \ \text{as} \ n,m \to \infty  .\]

\item Cauchy almost in $L_p$ if for each $\delta > 0$ there exists a measurable set $E_\delta$ with $\mu(E_\delta) < \delta$ such that 
\[\int_{E_\delta^c} |f_n - f_m|^p \, d\mu \to 0 \ \text{as} \ n,m \to \infty  .\]

\end{enumerate}

\end{definition}
If a sequence $(f_n)$ $\ap$-converges to $f$ then it is $\ap$-Cauchy. This follows from the inequality 
\[\int_{B_n \cap B_m} |f_n - f_m|^p \, d\mu \leq 2^{p-1} \int_{B_n } |f_n - f|^p  \, d\mu + 2^{p-1} \int_{B_m } |f - f_m|^p  \, d\mu  .\]
It is also clear that if $(f_n)$ converges to $f$ almost in $L_p,$ then $(f_n)$ is Cauchy almost in $L_p.$ \par
Since Cauchy sequences converge both in the $L_p$ and measure senses, one also expects the same to hold with both $\ap$-convergence and convergence almost in $L_p.$
First, it is proved that given a sequence $(f_n)$ that is Cauchy almost in $L_p,$ there exists a measurable function $f$ such that $(f_n)$ converges to $f$ almost in $L_p$. Then, with this result at hand, one will be able to prove that completeness in this sense also holds for the $\ap$-convergence.

\begin{theorem} \label{completeAp}
Let $(f_n)$ be a sequence of measurable functions. If $(f_n)$ is Cauchy almost in $L_p,$ then there exists a measurable function $f$ to which $(f_n)$ converges almost in $L_p.$
\end{theorem}
\begin{proof}
Assume that $(f_n)$ is Cauchy almost in $L_p.$ Then, arguing by contradiction, one can prove that $(f_n)$ is Cauchy in measure. This implies that there is a subsequence $(f_{k_n})$ and a measurable function $f$ such that $(f_{k_n})$ converges to $f$ almost everywhere. From that one can conclude that $(f_{k_n})$ converges to $f$ almost in $L_p$. Indeed, given $\delta > 0$ and $E_\delta$ with $\mu(E_\delta) < \delta$ such that 
\[ \int_{E_\delta^c} |f_{k_n} - f_{k_m}| \, d\mu \to 0 \ \text{as} \ n,m \to \infty  , \]
by Fatou's lemma one has:
\begin{equation*}
\begin{split}
\int_{E_\delta^c} |f_{k_n} - f| \, d\mu & = \int_{E_\delta^c} \lim\limits_{m \to \infty}|f_{k_n} - f_{k_m}| \, d\mu \\
& \leq \liminf\limits_{m \to \infty} \int_{E_\delta^c} |f_{k_n} - f_{k_m}| \, d\mu \\
& \to 0 \ \text{as} \ n \to \infty .
\end{split}
\end{equation*}
The result follows by a standard argument. 
\end{proof}

\begin{theorem} \label{completealphap}
If $(f_n)$ is a $\ap$-Cauchy sequence, then there exists a measurable function $f$ such that $f_n \apto f. $
\end{theorem}
\begin{proof}
Let $(f_n)$ be a $\ap$-Cauchy sequence. One can prove, in a similar fashion as in Proposition \ref{alphaimpliesalmostsubsequence}, that there exists a subsequence $(f_{k_n})$ which is Cauchy almost in $L_p$. Then, by the previous theorem there exists a measurable function $f$ such that $(f_{k_n})$ converges to $f$ almost in $L_p$. Hence, by Proposition \ref{almostLimpliesalpha}, $(f_{k_n})$ $\ap$-converges to $f$. Now let $\varepsilon > 0$. There exist $N_1,N_2 \in \mathbb{N}$ such that 
\[\forall n,m \geq N_1 \quad  \int_{B_n \cap B_m} |f_n - f_m|^p \, d\mu < \frac{\varepsilon}{2^{p}} ,  \qquad \forall n \geq N_2 \quad \int_{D_n} |f_{k_n} - f|^p \, d\mu < \frac{\varepsilon}{2^{p}}  , \]
where $(B_n),$ $(D_n)$ are the sequences of measurable sets associated with the $\ap$-Cauchy property of $(f_n)$ and the $\ap$-convergence of $(f_{k_n}),$ respectively.
For each $n \in \mathbb{N},$ let $C_n = B_n \cap D_n \cap B_{k_n}.$ Clearly, $\mu(C_n^c) \to 0$ as $n \to \infty.$ Finally,
for $n \geq N = \max\{N_1,N_2 \}$ one has
\begin{equation*}
\int_{C_n} |f_n - f|^p \, d\mu \leq 2^{p-1} \int_{B_n \cap B_{k_n}} |f_n - f_{k_n}|^p \, d\mu + 2^{p-1} \int_{D_n} |f_{k_n}-f|^p \, d\mu < \varepsilon ,
\end{equation*}
which completes the proof.
\end{proof}

\section{Convergence theorems} \label{convergencesection}

\begin{proposition}
Let $(f_n),f$ be measurable functions. Then 
\begin{equation*}
f_n \to f \  \text{in} \ L_p \quad \Leftrightarrow \quad
\begin{dcases} 
f_n \apto f , \\
 \forall (D_n) \ \Big( \mu(D_n) \to 0 \ \Rightarrow \ \int_{D_n} |f_n - f|^p \, d\mu \to 0 \Big) \ \text{as} \ n \to \infty  .
\end{dcases}
\end{equation*}
\end{proposition}
\begin{proof}
$ $\newline
$(\Rightarrow)$ This is clear. \\
$(\Leftarrow)$ Take $D_n = B_n^c,$ where $(B_n)$ is the sequence of measurable sets associated with the $\ap$-convergence of $(f_n).$
\end{proof}

\begin{theorem} \label{alphaconvergencethm}
Let $(f_n),f$ be measurable functions. Then 
\begin{equation*}
f_n \apto f  \quad \Leftrightarrow \quad
\begin{dcases} 
f_n \to f \ \text{in measure}  , \\ 
\exists \delta > 0 \  \int_{E_n^c(\delta)} |f_n - f|^p \, d\mu \to 0  \ \text{as} \ n \to \infty ,
\end{dcases}
\end{equation*}
where $E_n(\delta) = \{x \in X \ | \ |f_n(x) - f(x)| \geq \delta \}.$ 
\end{theorem}
\begin{proof}
$ $\newline
$(\Leftarrow)$ Simply consider $B_n = E_n^c(\delta).$ \\
$(\Rightarrow)$ Due to Proposition \ref{alphaimplymeasure}, one already knows that $\ap$-convergence implies convergence in measure. In order to prove the second condition, suppose, towards a contradiction, that for every $\delta > 0$ one has that $\int_{E_n^c(\delta)} |f_n - f|^p \, d\mu \not\to 0.$ Take $\delta = 1$ and set $E_n = E_n(1).$ Then 
\[\exists \varepsilon>0 \quad \forall n \in \mathbb{N} \quad \exists k_n \geq n \quad \int_{E_{k_n}^c} |f_{k_n} - f|^p \, d\mu \geq \varepsilon .\]
Moreover, there exists $N \in \mathbb{N}$ such that for $k \geq N$ one has 
\[\mu(B_k^c) < \frac{\varepsilon}{4}  , \quad \int_{B_k} |f_k - f|^p \, d\mu < \frac{\varepsilon}{4} .\] Thus,
\begin{equation*}
\begin{split}
\varepsilon & \leq \int_{E_{k_N}^c}|f_{k_N}-f|^p \, d\mu \\
& = \int_{E_{k_N}^c \cap B_{k_N}}|f_{k_N}-f|^p \, d\mu + \int_{E_{k_N}^c \cap B_{k_N}^c}|f_{k_N}-f|^p \, d\mu \\
& < \int_{B_{k_N}}|f_{k_N}-f|^p \, d\mu + \mu(B_{k_N}^c) \\
& < \varepsilon/2  ,
\end{split}
\end{equation*}
a contradiction. 
\end{proof}
\begin{lemma} \label{lemmameasure}
Suppose that $(f_n)$ converges to $f$ in measure. Then the following conditions are equivalent: 
\begin{enumerate}[(i)]
 \item $ \exists \delta > 0 \  \int_{E_n^c(\delta)} |f_n - f|^p \, d\mu \to 0 \ \text{as} \ n \to \infty,$ 
\item $\forall \delta > 0 \  \int_{E_n^c(\delta)} |f_n - f|^p \, d\mu \to 0 \ \text{as} \ n \to \infty,$
\end{enumerate}
where $E_n(\delta) = \{x \in X \ | \ |f_n(x) - f(x)| \geq \delta \}.$
\end{lemma}
\begin{proof} 
Suppose that there exists $\tilde{\delta} > 0$ such that \[\int_{E_n^c(\tilde{\delta})} |f_n - f|^p \, d\mu \to 0 \  \text{as} \ n \to \infty  . \] 
If $0 < \delta < \tilde{\delta},$ then $E_n^c(\delta) \subseteq E_n^c(\tilde{\delta})$ and hence 
\[\int_{E_n^c(\delta)} |f_n - f|^p \, d\mu \leq \int_{E_n^c(\tilde{\delta})} |f_n - f|^p \, d\mu \to 0 \  \text{as} \ n \to \infty  . \]
If $0 < \tilde{\delta} < \delta,$ then 
\begin{equation*}
\begin{split}
\int_{E_n^c(\delta)} |f_n - f|^p \, d\mu & = \int_{E_n^c(\delta) \cap E_n(\tilde{\delta})} |f_n - f|^p \, d\mu + \int_{E_n^c(\delta) \cap E_n^c(\tilde{\delta})} |f_n - f|^p \, d\mu \\
& < \delta^p \mu\big(E_n(\tilde{\delta})\big) + \int_{E_n^c(\tilde{\delta})} |f_n - f|^p \, d\mu \\
& \to 0 \ \text{as} \ n \to \infty  .
\end{split}
\end{equation*}
\end{proof}
The next concept, inspired by \cite{dudley, kelley}, serves as an attempt to organize some of the important properties that relevant notions of convergence satisfy (for instance, metric convergence notions):
\begin{definition}
A sequential convergence class consists of a pair $(M, \alpha),$ where $M$ is a set and $\alpha$ is a relation between sequences $(f_n)\subseteq M$ and elements $f \in M$ denoted by $f_n \xrightarrow{\alpha} f,$ such that:
\begin{enumerate}[(i)]
\item If $f_n= f \ \forall n \in \mathbb{N},$ then $f_n \xrightarrow{\alpha} f$,  \\ 
\item If $f_n \xrightarrow{\alpha} f$ and $(f_{k_n}) \subseteq (f_n),$ then $f_{k_n} \xrightarrow{\alpha} f$, \\
\item If $f_n \centernot{\xrightarrow{\alpha}} f,$ then there exists $(f_{k_n}) \subseteq (f_n)$ such that for any subsequence $(h_n) \subseteq (f_{k_n}),$  $h_n \centernot{\xrightarrow{\alpha}} f$, \\
\item If $f_n \xrightarrow{\alpha} f$ and $\forall n \in \mathbb{N} \ f_{n,k}\xrightarrow{\alpha} f_n,$ then 
\[\forall n \in \mathbb{N} \quad \exists k_n \geq n \quad  f_{n,k_n} \xrightarrow{\alpha} f .\]
\end{enumerate}
\end{definition}
\begin{theorem}
$(M(X), \alpha_p)$ is a sequential convergence class.
\end{theorem}
\begin{proof}
The first and second conditions are clear, so it suffices to prove the third and fourth. \par 
In order to prove the third condition, assume that $(f_n)$ is a sequence of measurable functions that does not $\ap$-converge to $f$. If $(f_n)$ does not converge in measure to $f$ then
\[\exists\delta > 0  \quad \exists \varepsilon > 0 \quad \forall n\in \mathbb{N} \quad \exists k_n \geq n \quad \mu\big(E_{k_n}(\delta) \big) \geq \varepsilon  ,\]
where \[E_{k_n}(\delta) = \{x \in X \ | \ |f_{k_n}(x) - f(x)| \geq \delta \}  .\] It is clear that for any subsequence $(h_n) \subseteq (f_{k_n}),$ $(h_n)$ does not converge to $f$ in measure, and hence $(h_n)$ does not $\ap$-converge to $f$. On the other hand, if $(f_n)$ converges in measure to $f$, since it does not $\ap$-converge to $f$, by Theorem \ref{alphaconvergencethm} and Lemma \ref{lemmameasure} it follows that
\[\exists\delta > 0  \quad \exists \varepsilon > 0 \quad \forall n\in \mathbb{N} \quad \exists k_n \geq n \quad \int_{E_{k_n}^c(\delta)} |f_{k_n} - f|^p \, d\mu \geq \varepsilon  . \]
Consequently, if $(h_n) \subseteq (f_{k_n}),$ then
\[\exists\delta > 0 \quad \int_{H_{n}^c(\delta)} |h_n - f|^p \, d\mu \not\to 0  ,\]
where \[H_n(\delta) = \{x \in X \ | \ |h_n(x) - f(x)| \geq \delta \} , \]
hence $(h_n)$ does not $\ap$-converge to $f$. This proves that $(M(X), \ap)$ satisfies the third condition of a sequential convergence class.
\par
Regarding the fourth condition, let $(B_n)$ be the sequence of measurable sets associated with the $\ap$-convergence of $(f_n)$ and let, for each $n \in \mathbb{N},$ $(D_{n,k})$ be the sequence of measurable sets associated with the $\ap$-convergence of $(f_{n,k}).$ Then, for each $n \in \mathbb{N}$ there exists $k_n \geq n$ such that \[\int_{D_{n,k_n}} |f_{n,k_n}-f_n|^p \, d\mu < \frac{1}{2^n}  . \] Set $C_n = B_n \cap D_{n,k_n}$ and let $\varepsilon > 0.$ There exists $N \in \mathbb{N}$ such that whenever $n \geq N$ one has 
\[\int_{B_n} |f_n-f|^p \, d\mu < \frac{\varepsilon}{2^p}  ,\]

\[\int_{D_{n,k_n}} |f_{n,k_n}-f_n|^p \, d\mu < \frac{\varepsilon}{2^p}  . \]
Thus, for $n \geq N,$
\[
\int_{C_n} |f_{n,k_n}-f|^p \, d\mu \leq 2^{p-1}\int_{D_{n,k_n}} |f_{n,k_n}-f_n|^p \, d\mu + 2^{p-1}\int_{B_n} |f_n-f|^p \, d\mu < \varepsilon  ,
\]
which concludes the proof.

\end{proof}

\section{Preservation of convergence under composition} \label{preservationsection}
\begin{definition}
Let $\mathfrak{m}$ be a notion of convergence for sequences of measurable functions. A function $\varphi : \mathbb{R} \to \mathbb{R}$ is said to preserve $\mathfrak{m}$-convergence if given a measure space $(X,\mathbb{X}, \mu)$ and a sequence $(f_n)$ of $\mathbb{X}$-measurable functions that $\mathfrak{m}$-converges to $f,$ then the sequence $(\varphi(f_n))$ $\mathfrak{m}$-converges to $\varphi(f)$.
\end{definition}

In \cite{bartlepreserv} it is proved that:

\begin{enumerate}[(i)]
\item $\varphi$ preserves almost everywhere convergence if and only if $\varphi$ is continuous, 
\item $\varphi$ preserves uniform convergence, almost uniform convergence or convergence in measure if and only if $\varphi$ is uniformly continuous.
\end{enumerate}
Here, one obtains a similar result as (ii) but for the $L_p$-convergence, convergence almost in $L_p$ and $\ap$-convergence. Precisely:
\begin{theorem} \label{PreservationThm}
A function $\varphi$ preserves $L_p$-convergence, convergence almost in $L_p$ or $\ap$-convergence if and only if $\varphi$ is Lipschitz continuous. 
\end{theorem}
Before proving the theorem, one needs a local characterization of Lipschitz continuity.
\begin{lemma}
A function $\varphi : \mathbb{R} \to \mathbb{R}$ is Lipschitz continuous if and only if 
\begin{equation} \label{phiconditon}
\exists \delta>0 \quad \exists K > 0 \quad \forall a,b \in \mathbb{R} \quad |a-b|<\delta \Rightarrow |\varphi(a) - \varphi(b)| \leq K|a-b|  .
\end{equation}
\end{lemma}
\begin{proof} It is clear that a Lipschitz continuous function satisfies (\ref{phiconditon}). Assume that $\varphi : \mathbb{R} \to \mathbb{R}$ satisfies (\ref{phiconditon}). Let $a,b \in \mathbb{R}$ and assume without loss of generality that $a < b.$ Let $\tilde{\delta} = \delta / 2$ and set $d = |a - b|.$ There exists $N \in \mathbb{N}_0$ such that $d = N \tilde{\delta} + \varepsilon$ for some $\varepsilon \in [0, \tilde{\delta}[$ (take $N$ as the integer part of $d/\tilde{\delta}$). Now, let 
\[x_n = a + n \tilde{\delta}  , \quad n = 0, 1, \ldots, N  , \]
\[x_{N+1} = a + N\tilde{\delta} + \varepsilon = b  . \]
Then, the triangle inequality together with (\ref{phiconditon}) yields
\begin{equation*}
\begin{split}
|\varphi(a) - \varphi(b)| & \leq \sum_{n=0}^{N-1} |\varphi(x_n)-\varphi(x_{n+1})| + |\varphi(x_N) - \varphi(b)| \\
                         & \leq \sum_{n=0}^{N-1} K|x_n - x_{n+1}| + K|x_N - b| \\ 
                         & = \sum_{n=0}^{N-1} K \tilde{\delta} + K \varepsilon \\ 
                         & = K(N\tilde{\delta} + \varepsilon) \\
                         & = K|a - b|  ,
\end{split}
\end{equation*}
which completes the proof.
\end{proof}

\begin{proof}[Proof of Theorem \ref{PreservationThm}]
$ $\newline
$(\Leftarrow)$ Suppose that $\varphi$ is Lipschitz continuous and that $f_n \to f$ in $L_p.$ Then, there exists $K > 0$ such that \[\forall a,b \in \mathbb{R} \quad |\varphi(a) - \varphi(b)| \leq K|a-b|  , \] 
and \[\int_{X} |f_n - f|^p \, d\mu \to 0 \ \text{as} \ n \to \infty  . \]
Thus, \[\int_{X} |\varphi(f)_n - \varphi(f)|^p \, d\mu \leq K^p\int_{X} |f_n - f|^p \, d\mu \to 0 \ \text{as} \ n \to \infty  , \]
that is, \[\varphi(f_n) \to \varphi(f) \ \text{in} \ L_p  .\] 
In a similar fashion, if $f_n \to f$ almost in $L_p$ or  $f_n \apto f$ then $\varphi(f_n) \to \varphi(f)$ almost in $L_p$ or $\varphi(f_n) \apto \varphi(f),$ respectively.\\
$(\Rightarrow)$ 
Suppose that $\varphi$ does not satisfy (\ref{phiconditon}) and choose sequences $(a_n),(b_n) \subseteq \mathbb{R}$ such that, for each $n \in \mathbb{N},$ \[0<|a_n - b_n| < \frac{1}{n^{1/p}}  , \quad |\varphi(a_n)- \varphi(b_n)| > n^{1/p}|a_n - b_n|  . \]
Let $(X,\mathbb{X},\mu)=([0,\infty), \mathcal{B},\mathcal{L}),$ define $f : X \to \mathbb{R}$ by
\begin{equation*} \label{preserv_f}
f(x) = \begin{cases}
            b_1  , \quad \text{if} \ 0 \leq x < \frac{1}{|a_1-b_1|^p}  , \\
             b_n  , \quad \text{if} \ \sum_{k=1}^{n-1} \frac{1}{|a_k-b_k|^p} + \frac{n-1}{n|a_n-b_n|^p} \leq x < \sum_{k=1}^{n} \frac{1}{|a_k-b_k|^p}  , \quad n \in \mathbb{N}\setminus{\{1\}} , \\ 
              0  , \quad \text{otherwise}  .
\end{cases} 
\end{equation*}
and, for each $n \in \mathbb{N} \setminus \{1\},$ let $f_n : X \to \mathbb{R}$ be given by
\begin{equation*} \label{preserv_fn}
f_n(x) = \begin{cases}
             a_n  , \quad \sum_{k=1}^{n-1} \frac{1}{|a_k-b_k|^p} + \frac{n-1}{n|a_n-b_n|^p} \leq x < \sum_{k=1}^{n} \frac{1}{|a_k-b_k|^p}  , \\ 
              f(x)  , \quad \text{otherwise}  .
\end{cases}
\end{equation*}
It holds that
\[\int_X |f_n - f|^p \, d\mu = \frac{1}{n} \to 0 \ \text{as} \ n \to \infty  ,\]
that is, $(f_n)$ converges to $f$ in $L_p$. Hence $f_n \to f$ almost in $L_p$ and $f_n \apto f.$ It remains to be shown that $(\varphi(f_n))$ does not $\ap$-converge to $\varphi(f)$ (and hence it does not converge neither almost in $L_p$ nor in $L_p$). To that end, let $(B_n)$ be any sequence of measurable sets such that $\mu (B_n^c) \to 0$ as $n \to \infty,$ set \[I_n = \Bigg[\sum_{k=1}^{n-1} \frac{1}{|a_k-b_k|^p} + \frac{n-1}{n|a_n-b_n|^p},   \sum_{k=1}^{n} \frac{1}{|a_k-b_k|^p}\Bigg[  ,\] and notice that 
\begin{equation*}
\begin{split}
\int_{B_n} |\varphi(f_n) - \varphi(f)|^p \, d\mu & = \int_{B_n \cap I_n} |\varphi(a_n) - \varphi(b_n)|^p \, d\mu \\
& > n|a_n - b_n|^p \mu(B_n \cap I_n) \\
& = n|a_n - b_n|^p \big( \mu(I_n) - \mu(B_n^c \cap I_n) \big) \\
& = 1 - n|a_n - b_n|^p \mu(B_n^c \cap I_n)  .
\end{split}
\end{equation*}
Now let $N \in \mathbb{N}$ be such that $\mu(B_n^c) < 1/2$ whenever $n \geq N.$ Thus, for $n \geq N,$
\begin{equation*}
\begin{split}
\int_{B_n} |\varphi(f_n) - \varphi(f)|^p \, d\mu & > 1 - n|a_n - b_n|^p \mu(B_n^c \cap I_n) \\
& > 1 - \mu(B_n^c) \\
& > 1/2 ,
\end{split}
\end{equation*}
therefore $(\varphi(f_n))$ does not $\ap$-converge to $\varphi(f)$ and the proof is complete.
\end{proof}


\section*{Acknowledgements}
The authors would like to thank Professor Athanasios Tzavaras for raising the question that originated the development of this work. 


\end{document}